\documentclass[11pt]{article}

\usepackage{latexsym,amsfonts,amsmath,amssymb,amsthm,url,graphics,psfrag,amscd}
\usepackage[english]{babel}
\usepackage[latin1]{inputenc}
\usepackage[T1]{fontenc}
\usepackage{graphics}
\usepackage{graphicx}
\usepackage{color}

\newcommand{\defeq}{\overset{\hbox{\tiny{def}}}{=}}
\renewcommand{\P}{\mathbf{P}}

\newcommand{\N}{\mathbb{N}}
\newcommand{\E}{\mathbb{E}}
\newcommand{\Z}{\mathbb{Z}}

\newcommand{\lin}{\left[\kern-0.15em\left[}
\newcommand{\rin} {\right]\kern-0.15em\right]}
\newcommand{\linf}{[\kern-0.15em [}
\newcommand{\rinf} {]\kern-0.15em ]}
\newcommand{\ilin}{\left]\kern-0.15em\left]}
\newcommand{\irin} {\right[\kern-0.15em\right[}

\newtheorem{lem}{Lemma}[section]

\newtheorem{prop}[lem]{Proposition}
\newtheorem{theo}[lem]{Theorem}

\begin{document}

\title{\textbf{Recurrence for vertex-reinforced random walks on $\Z$ with weak reinforcements.}}

\author{
\textsc{Arvind Singh}\footnote{Département de Mathématiques,
Université Paris XI, France. E-mail: arvind.singh@math.u-psud.fr}}
\date{}
\maketitle

\vspace*{0.2cm}

\begin{abstract}
We prove that any vertex-reinforced random walk on the integer lattice with non-decreasing reinforcement sequence $w$ satisfying $w(k) = o(k^{\alpha})$ for some $\alpha <1/2$ is recurrent. This improves on previous results of Volkov \cite{V2} and Schapira \cite{Sch}.
\end{abstract}

\bigskip
{\small{
 \noindent{\bf Keywords. } Self-interacting random walk; reinforcement; recurrence and transience

\bigskip
\noindent{\bf A.M.S. Classification. } 60K35

\section{Introduction}

In this paper, we consider a one-dimensional vertex-reinforced random walk (VRRW) with non-decreasing weight sequence $w:\N\to (0,\infty)$, that is a stochastic process  $X = (X_n)_{n\geq0}$ on $\Z$, starting from $X_0 = 0$, with transition probabilities:
$$
\P\{X_{n+1} = X_n \pm 1 \,|\, \mathcal{F}_n\} = \frac{w(Z_n(X_n \pm 1))}{w(Z_n(X_n+1)) + w(Z_n(X_n - 1))}
$$
where $\mathcal{F}_n \defeq \sigma(X_1,\ldots,X_n)$ is the natural filtration of the process and $Z_n(x)\defeq \#\{0\leq k\leq n,X_k = x\}$ is the local time of $X$ on site $x$ at time $n$. This process was first introduced by Pemantle in \cite{P}
and then studied in the linear case $w(k) = k+1$ by Pemantle and Volkov in \cite{PV}. They proved the surprising fact that the walk visits only finitely many sites. This result was subsequently improved by Tarrès \cite{T1,T2} who showed that the walk eventually gets stuck on exactly $5$ consecutive sites almost surely. When the reinforcement sequence grows faster than linearly, the walk still gets stuck on a finite set but whose cardinality may be smaller than $5$, see \cite{BSS,V2} for details. On the other hand, Volkov \cite{V2} proved that for sub-linearly growing weight sequences of order $n^\alpha$ with $\alpha<1$, the walk necessarily visits infinitely many sites almost-surely. Later, Schapira \cite{Sch} improved this result showing that, when $\alpha < 1/2$, the VRRW is either transient or recurrent. The main result of this paper is to show that the walk is, indeed, recurrent.

\begin{theo}\label{maintheo} Assume that the weight sequence satisfies $w(k) = o(k^\alpha)$ for some $\alpha< 1/2$. Then $X$ is recurrent \emph{i.e.} it visits every site infinitely often almost-surely.
\end{theo}
Let us mention that, simultaneously with the writing of this paper, a similar result was independently obtained by Chen and Kozma \cite{CK} who proved recurrence for the VRRW with weights of order $n^\alpha$, $\alpha<1/2$, using a clever martingale argument combined with previous local time estimates from Schapira \cite{Sch}. The argument in this paper, while also making use of a martingale, is self-contained and does not rely upon previous results of Volkov \cite{V2} or Schapira \cite{Sch}. In particular, we do not require any assumption on the regular variation of the weight function $w$.

\section{A martingale}
Obviously, multiplying the weight function by a positive constant does not change the process $X$. Thus, we now assume without loss of generality that $w(0) = 1$. We define the two-sided sequence $(a_x)_{x\in\Z}$ by
$$
a_x \defeq \left\{
\begin{array}{ll}
1 - \frac{1}{(x+2)^{1+\varepsilon}} & \hbox{for $x\geq 0$}\\
\frac{1}{2} & \hbox{for $x < 0$}
\end{array}
\right.
$$
where $\varepsilon>0$ will be chosen later during the proof of the theorem. Define also
$$
A_k \defeq \prod_{x = -k}^{\infty} a_x \in (0,1).
$$
We construct from $X$ two processes $(M_n)_{n\geq 0}$ and $(\Delta_n(z),\, z < X_n)_{n\geq 0}$ in the following way:
\begin{enumerate}
\item Initially set $M_0 \defeq 0$ and $\Delta_0(z) \defeq 1$ for all $z< 0 = X_0$.
\item By induction, $M_n$ and $(\Delta_n(z),\, z < X_n)$ having been constructed,
\begin{itemize}
\item if $X_n = x$ and $X_{n+1} = x - 1$, then
$$
  \begin{array}{lcll}
    M_{n+1} &\defeq& M_n - a_x\Delta_n(x-1) &\hbox{}\\
    \Delta_{n+1}(z) &\defeq& \Delta_{n}(z) \quad \hbox{for $z < x-1$,} \\
  \end{array}
$$
\item if $X_n = x$ and $X_{n+1} = x + 1$, then
$$
  \begin{array}{lcll}
    \hspace{2.7cm}M_{n+1} &\defeq& M_n + a_x\Delta_n(x-1)\frac{w(Z_n(x-1))}{w(Z_n(x+1))} &\hbox{} \\
    \hspace{2.7cm}\Delta_{n+1}(z) &\defeq&
\left\{
\begin{array}{ll}
 \Delta_{n}(z) & \hbox{for $z < x$,} \\
 a_x\Delta_n(x-1)\frac{w(Z_n(x-1))}{w(Z_n(x+1))} & \hbox{for $z = x$.}
\end{array}
\right.
  \end{array}
$$
\end{itemize}
\end{enumerate}
Note that the quantities $\Delta$ have a simple interpretation: for any $n$ and $z<X_n$, the value $\Delta_n(z)$ is positive and corresponds to the increments of $M_n$ the last time before $n$ that the walk $X$ jumped from site $z$ to site $z+1$ (with the convention $\Delta_n(z) = 1$ for negative $z$ if no such jumps occurred yet). By extension, we also define $\Delta_n \defeq \Delta_n(X_n)$ at the current position as the "would be" increment of  $M_n$ if $X$ makes its next jumps to the right (at time $n+1$) \emph{i.e.}
$$
\Delta_n \defeq a_{X_n}\Delta_n(X_n - 1)\frac{w(Z_n(X_n-1))}{w(Z_n(X_n+1))}.
$$
We will also use the notation $\tau_y$ to denote the hitting time of site $y$,
$$
\tau_{y} \defeq \inf\{n\geq 0,\, X_n = y\} \in [0,\infty].
$$

\medskip

\begin{prop} The process $M$ is an $\mathcal{F}_n$-martingale and, for $n\geq 0$, we have
\begin{equation}\label{eq1}
M_n = \sum_{i=0}^{n-1} \mathbf{1}_{\{X_{i+1} = X_i + 1\}}\left(1 - a_{X_{i +1}}\mathbf{1}_{\{\exists j \in (i,n],\, X_j = X_i\}}\right)\Delta_i + \frac{1}{2}\inf_{i\leq n}X_i
\end{equation}
In particular, for $y=1,2,\dots$, the process $M_{n\wedge \tau_{-y}}$ is bounded below by $-y/2$, hence it converges a.s.
\end{prop}

\begin{proof} Since $\Delta_n(\cdot)$ and $Z_n(\cdot)$ are $\mathcal{F}_n$-measurable, by definition of $M$,
\begin{multline*}
\E[M_{n+1}\,|\,\mathcal{F}_n] \\
\begin{aligned}
& = \E\left[ M_n + a_{X_n}\Delta_n(X_n-1) \left(\frac{w(Z_n(X_n-1))}{w(Z_n(X_n+1))}\mathbf{1}_{\{X_{n+1} = X_n + 1\}} - \mathbf{1}_{\{X_{n+1} = X_n - 1\}}\right)  \Big|\,\mathcal{F}_n\right]\\
&=  M_n + a_{X_n}\Delta_n(X_n-1)\left( \frac{w(Z_n(X_n-1))}{w(Z_n(X_n+1))}\P\{X_{n+1} = X_n + 1 \,|\, \mathcal{F}_n \} - \P\{X_{n+1} = X_n - 1 \,|\, \mathcal{F}_n \} \right)\\
&=  M_n
\end{aligned} 
\end{multline*}
thus $M$ is indeed a martingale. Furthermore, by construction, at each time $i$ when the process $X$ crosses an edge $\{x,x+1\}$ from left to right, the process $M$ increases by $\Delta_i = \Delta_{i+1}(x) >0$. If at some later time, say $j > i$, $X$ crosses this edge again (and thus in the other direction), the martingale decreases by $a_{x+1} \Delta_{j}(x) = a_{x+1} \Delta_i$. Moreover, by convention $\Delta_0(z) = 1$ and $a_z = \frac{1}{2}$ for $z<0$ so that $M$ decreases by $\frac{1}{2}$ each time it crosses a new edge of the negative half line for the first time. Putting these facts together, we get exactly \eqref{eq1}. Finally, since $a_z < 1$ for any $z\in\Z$, each term in the sum \eqref{eq1} is positive, hence $M_{n\wedge\tau_{-y}}$ is bounded below by $\frac{1}{2}\inf_{i\leq n\wedge\tau_{-y}} X_i \geq -y/2$.
\end{proof}

\begin{prop}\label{minodelta} Let $y >0$. For $n\le \tau_{-y}$, we have
\begin{equation}\label{eq2}
\Delta_n(z)\ge \frac{A_y}{w(Z_n(z))w(Z_n(z+1))}\quad\hbox{for any $-y\le z \le X_n$.}
\end{equation}
\end{prop}

\begin{proof} We prove by induction on $n$ that for $n\le \tau_{-y}$,
\begin{equation}\label{eq3}
\Delta_n(z)\ge \frac{\prod_{i= -y}^z a_i}{w(Z_n(z))w(Z_n(z+1))} \hbox{ for any $-y\le z \le X_n$}.
\end{equation}
Recalling that $w(k)\ge 1$ and $a_k\le 1$ for any $k$, it is straightforward that \eqref{eq3} holds for $n=0$. Now, assume the result for $n$ and consider the two cases:
\begin{itemize}
\item If $X_{n+1}=X_n-1$. Then for any $-y\le z \le X_{n+1}$, we have $\Delta_{n+1}(z)= \Delta_n(z)$ whereas $w(Z_{n+1}(z))\ge w(Z_{n}(z))$. Thus \eqref{eq3} holds for $n+1$.
\item If $X_{n+1}=X_n+1$. Again, we have $\Delta_{n+1}(z)= \Delta_n(z)$ for any $-y\le z \le X_{n}$. It remains to check that $\Delta_{n+1}(X_{n+1})$ satisfies  the inequality:
\begin{eqnarray*} \Delta_{n+1}(X_{n+1}) & = &
\Delta_{n+1} = a_{X_{n+1}}\Delta_{n+1}(X_n)\frac{w(Z_{n+1}(X_n))}{w(Z_{n+1}(X_{n+1}+1))}\\
&\geq & a_{X_{n+1}}\frac{\prod_{i= -y}^{X_n} a_i}{w(Z_n(X_n))w(Z_n(X_n+1))}\frac{w(Z_{n+1}(X_n))}{w(Z_{n+1}(X_{n+1}+1))}\\
&\geq & \frac{\prod_{i= -y}^{X_{n+1}} a_i}{w(Z_{n+1}(X_{n+1}))w(Z_{n+1}(X_{n+1}+1))}.
\end{eqnarray*}
\end{itemize}
\end{proof}

We can now recover, with our assumptions on $w$, Volkov's result \cite{V2} stating that the walk does not get stuck on any finite interval.
\begin{prop} For any $y>0$, we have
$$
\limsup_n X_n = +\infty \quad\hbox{on the event $\{\tau_{-y} = \infty\}$.}
$$
\end{prop}

\begin{proof}
On $\{\tau_{-y}=\infty\}$, the combination of \eqref{eq1} and Proposition \ref{minodelta} give
\begin{equation}\label{eq4}
M_n\geq  \sum_{i=0}^{n-1} \mathbf{1}_{\{X_{i+1} = X_i + 1\}}\left(1 - a_{X_{i +1}}\mathbf{1}_{\{\exists j >i,\, X_j = X_i\}}\right)\frac{A_y}{w(Z_i(X_i))w(Z_i(X_i+1))}  - \frac{y-1}{2}.
\end{equation}
Denoting by $e_n=(s_n,s_{n}+1)$ the edge which has been most visited at time $n$, we deduce that on the event $\{\tau_{-y}=\infty\}$,
$$M_{n}\geq  Z_n(e_n)\left(1 - a_{s_n+1}\right)\frac{A_y}{w(Z_n(s_n))w(Z_n(s_n+1))}   - \frac{y-1}{2},
$$
where $Z_n(e_n)$ denotes the number of times the edge $e_n$ has been crossed from left to right before time $n$.
Using that $\max(Z_n(s_n),Z_n(s_n+1))\leq 2Z_n(e_n)$ and that $w(k)=o(\sqrt{k})$ and that $M_{n\wedge \tau_{-y}}$ converges, we conclude that
on $\{\tau_{-y}=\infty\}$, either $Z_n(e_n)$ remains bounded or $a_{s_n+1}$ takes values arbitrarily close to $1$. In any case, this means that $X$ goes arbitrarily far to the right hence $\limsup_n X_n = +\infty$.
\end{proof}

\section{Proof of theorem \ref{maintheo}}

Fix $y>0$ and consider the event $\mathcal{E}_y = \{ \inf_n{X_n} = -y+1\}$. Pick $v>0$ and define $N_z$ to be the number of jumps of $X$ from site $z$ to site $z+1$ before time $\tau_v$ (according to the previous proposition $\tau_v$ is finite on $\mathcal{E}_y$ so all the $N_z$ are finite). From \eqref{eq4}, grouping together the contributions to $M$ of each edge $(z,z+1)$, we get, on $\mathcal{E}_y$,
\begin{eqnarray*}\label{eq5}
M_{\tau_v}&\geq& A_y \sum_{z=-y+1}^{v-1} \frac{1+(N_z-1)(1-a_v)}{w(Z_{\tau_v}(z))w(Z_{\tau_v}(z+1))}   - \frac{y-1}{2}\\
&\geq & A_y \sum_{z=-y+1}^{v-1} \frac{1+(N_z-1)(1-a_v)}{w(N_{z-1}+N_z)w(N_{z}+N_{z+1})}   - \frac{y-1}{2}\\
&\geq & A_y \sum_{z=-y+1}^{v-1} \frac{\frac{1}{2}+ N_z(1-a_v)}{w(N_{z-1}+N_z)w(N_{z}+N_{z+1})}   - \frac{y-1}{2}\\
&\geq & C A_y\sum_{z=-y+1}^{v-1} \frac{\frac{1}{2}+\frac{N_z}{(v+2)^{1+\varepsilon}}}{(N_{z-1}+N_z)^\alpha (N_{z}+N_{z+1})^\alpha}   - \frac{y-1}{2}
\end{eqnarray*}
where $C>0$ and $\alpha < 1/2$ only depend on the weight function $w$. Finally, lemma \ref{analem} below states that if we choose $\varepsilon > 0$ small enough, the sum above becomes arbitrarily large as $v$ tends to infinity. On the other hand, we also know that $M$ converges on this event so necessarily $\P\{\mathcal{E}_y\} = 0$. Since this result holds for any $y>0$, we get  $\inf X_n = -\infty$ a.s. By symmetry, $\sup X_n = +\infty$ a.s. which implies that the walk visits every site of the integer lattice infinitely often almost surely.

\section{An analytic lemma}

\begin{lem}\label{analem} For any $0 < \alpha < \frac{1}{2}$, there exists $\varepsilon >0$ such that
\begin{equation}\label{eq6}
\limsup_{K\to\infty}  \underset{(b_0,\ldots,b_{K})\in [1,\infty)^{K+1}}{\inf} \sum_{i=0}^{K}\frac{\frac{1}{2} + \frac{b_i}{(K+2)^{1+\varepsilon}}}{(b_{i-1} + b_i)^\alpha(b_i + b_{i+1})^\alpha} = \infty
\end{equation}
(with the convention $b_{-1} = b_{K+1} = 0$).
\end{lem}

\begin{proof} The idea is to group the $b_i$'s into packets with respect to their value. Consider a reordering of the $b_i$'s:
$$
\tilde{b}_0 \geq \tilde{b}_2 \geq \ldots \geq \tilde{b}_K.
$$
Fix a positive integer $l$ and group these numbers into $l+1$ packets
$$
\underset{\hbox{\tiny{packet $\mathcal{P}_1$}}}{\underbrace{\tilde{b}_0, \ldots, \tilde{b}_{K_1}}},
\underset{\hbox{\tiny{packet $\mathcal{P}_2$}}}{\underbrace{\tilde{b}_{K_1 + 1}, \ldots, \tilde{b}_{K_2}}}, \ldots,
\underset{\hbox{\tiny{packet $\mathcal{P}_l$}}}{\underbrace{\tilde{b}_{K_{l-1}}, \ldots, \tilde{b}_{K_l}}},
\underset{\hbox{\tiny{packet $\mathcal{P}_{l+1}$}}}{\underbrace{\tilde{b}_{K_l + 1}, \ldots, \tilde{b}_{K}}}.
$$
We can choose the $K_i$'s growing geometrically such that the sizes of the packets satisfy
\begin{equation}\label{eq7}
\#\mathcal{P}_1 \geq \frac{K}{4^l} \quad \hbox{and} \quad \#\mathcal{P}_i \geq 3(\# \mathcal{P}_1 + \ldots +  \#\mathcal{P}_{i-1}).
\end{equation}
We now regroup each term of the sum  \eqref{eq6} according to which packet the central $b_i$ (the one appearing in the numerator) belongs. Assume by contradiction that the sum \eqref{eq6} is bounded, say by $A$.

We first consider only the terms corresponding to packet $\mathcal{P}_1$. Since there are at least $\frac{K}{4^l}$ terms, we obtain the inequality
$$
A \geq \frac{K}{4^l}\frac{\frac{1}{2} + \frac{\tilde{b}_{K_1}}{(K+2)^{1+\varepsilon}}}{(2\tilde{b}_0)^\alpha(2\tilde{b}_0)^\alpha}
\geq C \frac{\tilde{b}_{K_1}}{K^{\varepsilon}\tilde{b}_0^{2\alpha}}
$$
where the constant $C$ does not depend on $K$ or on the sequence $(b_i)$. We now deal with packets $k = 2,\ldots,l$. Thanks to \eqref{eq7} and since every denominator in \eqref{eq6} involves only two $b_j$ other than the one appearing in the numerator, there are least one third of the terms belonging to packet $\mathcal{P}_k$ that do not contain any $b_j$ from a packet with smaller index (\emph{i.e.} with larger value). So, there is at least $\frac{K}{4^l}$ such terms for which we  can get a lower bound the same way we did for packet $\mathcal{P}_1$. We deduce that, summing over the terms corresponding to packet $\mathcal{P}_k$,
\begin{equation}\label{eq8}
\tilde{b}_{K_{k}} \leq C K^\varepsilon \tilde{b}^{2\alpha}_{K_{k-1}} \quad \hbox{for $i=1,\dots,l$,}
\end{equation}
with the convention $K_0 = 0$ and where $C$ again does not depend on $K$ or $(b_i)$. Finally, we obtain a lower bound for the sum of the terms belonging to the last packet $\mathcal{P}_{l+1}$ by taking $\frac{1}{2}$ as the lower bound for the numerator, and considering only the terms for which no $b_i$'s from any other packet appear in the denominator (again, there are at least $\frac{K}{4^l}$ such terms). This give the inequality
\begin{equation}\label{eq9}
K \leq C \tilde{b}_{K_l}^{2\alpha}.
\end{equation}
Combining \eqref{eq8} and \eqref{eq9}, we get by induction that for some constant $C$ (depending on $l$),
$$
K \leq C K^{\varepsilon(2\alpha + (2\alpha)^2 + \ldots + (2\alpha)^l)} \tilde{b}_0^{(2\alpha)^{l+1}} \leq C K ^{\frac{\varepsilon}{1-2\alpha}}\tilde{b}_0^{(2\alpha)^{l+1}}.
$$
For $\varepsilon$ small enough such that $\frac{\varepsilon}{1-2\alpha} \leq \frac{1}{2}$ we obtain
$$
\tilde{b}_0 \geq \frac{1}{C} K^{\frac{1}{2(2\alpha)^{l+1}}}.
$$
Recalling that the sum \eqref{eq6} contains the term corresponding to $\tilde{b}_0$ but is also, by assumption, bounded above by $A$, we find
$$
A \geq \frac{\frac{1}{2} + \frac{\tilde{b}_{0}}{(K+2)^{1+\varepsilon}}}{(2\tilde{b}_0)^\alpha(2\tilde{b}_0)^\alpha} \geq  C K^{\frac{1-2\alpha}{2(2\alpha)^{l+1}} - 1 - \varepsilon}.
$$
Finally, we choose $l$ large enough such that $\frac{1-2\alpha}{2(2\alpha)^{l+1}} - 1 - \varepsilon > 0$ and we get a contradiction by letting $K$ tends to infinity.
\end{proof}

\end{document}